\documentclass[12pt]{amsart}

\usepackage[english]{babel}
\usepackage{mathrsfs,amssymb}
\usepackage{mathtools}
\usepackage[colorlinks, citecolor = blue]{hyperref}

\usepackage{amsmath}

\usepackage[shortlabels]{enumitem}
\setlist[itemize]{leftmargin=25pt}
\setlist[enumerate]{leftmargin=25pt}

\usepackage[a4paper,left=3cm,right=3cm]{geometry}

\newtheorem{theorem}{Theorem}[section]

\newtheorem{prop}[theorem]{Proposition}

\theoremstyle{definition}

\theoremstyle{remark}

\numberwithin{equation}{section}
\usepackage[colorinlistoftodos,prependcaption,textsize=small]{todonotes}

\DeclareMathOperator*{\esssup}{ess\,sup}

\let \la=\lambda

\let \o=\omega
\let \a=\alpha
\let \f=\varphi

\let \O=\Omega

\allowdisplaybreaks

\begin{document}
\title[Sparse domination of Calder\'on--Zygmund operators]
{Sparse domination of Calder\'on--Zygmund operators by mean oscillations}

\author[A.K. Lerner]{Andrei K. Lerner}
\address[A.K. Lerner]{Department of Mathematics,
Bar-Ilan University, 5290002 Ramat Gan, Israel}
\email{lernera@math.biu.ac.il}

\thanks{The author was supported by ISF grant no. 1035/21.}

\begin{abstract}
We show that if $T$ is a Dini-continuous Calder\'on--Zygmund operator
satisfying $T(1)=0$, then the usual sparse domination for $T$ can be
sharpened by replacing local averages with local mean oscillations. This
extends a result of Benea and Bernicot for smoother kernels to the more
general Dini-continuous setting.

As an application, we characterize the Calder\'on--Zygmund
operators for which a pointwise Sobolev-type inequality holds: this is the case if and
only if $T(1)\in L^\infty$. This answers a recent question of Hoang, Moen and P\'erez.
\end{abstract}


\maketitle

\section{Introduction}

Let $T$ be a Calder\'on--Zygmund operator on ${\mathbb R}^n$.
A well-known pointwise sparse domination theorem asserts that, for every
integrable compactly supported function~$f$, there exists a sparse family
${\mathcal S}$ such that
\begin{equation}\label{spd}
|Tf(x)|\lesssim \sum_{Q\in{\mathcal S}}\langle |f|\rangle_Q\chi_Q(x)
\end{equation}
for almost every $x\in{\mathbb R}^n$; see, e.g., \cite{LO20} for a short proof.
Here $\langle f\rangle_Q:=|Q|^{-1}\int_Q f$, and sparseness means that there
exists $\eta\in(0,1)$, depending only on $n$, such that each $Q\in{\mathcal S}$
contains a measurable set $E_Q$ with $|E_Q|\ge\eta |Q|$, while the sets
$E_Q$ are pairwise disjoint.

The literature on sparse domination is vast; for background, motivation,
and an extensive bibliography, see \cite{LLO22}.

In \cite{BB18}, Benea and Bernicot showed that, for a subclass of
Calder\'on--Zygmund operators satisfying $T(1)=0$, the averages
$\langle |f|\rangle_Q$ in \eqref{spd} can be replaced by the mean oscillations
$$
\Omega(f;Q):=|Q|^{-1}\int_Q |f-\langle f\rangle_Q|.
$$
A careful examination of the proof in \cite{BB18} shows that the argument works
under a $\log^2$-Dini assumption on the kernel. More precisely, if $\omega$
denotes the modulus of continuity appearing in the smoothness condition for the
kernel, then the method requires
$
\int_0^1 \omega(t)\log^2\frac{1}{t}\,\frac{dt}{t}<\infty .
$

In this note, we show that the same result holds under the usual Dini
assumption
$
\int_0^1 \omega(t)\,\frac{dt}{t}<\infty .
$
Our proof is close in spirit to the recent work of Conde-Alonso--Lorist--Rey
\cite{CLR26} on cancellative sparse domination, where a different improvement
of \eqref{spd} was obtained by replacing the averages $\langle |f|\rangle_Q$
with cancellation-sensitive quantities of another type.

Our first result is the following.

\begin{theorem}\label{mr}Let $T$ be a Calder\'on--Zygmund operator satisfying $T(1)=0$.
For every compactly supported function $f\in L^1({\mathbb R}^n)$, there exists an $\eta_n$-sparse family ${\mathcal S}$ such that
$$|Tf(x)|\lesssim \sum_{Q\in {\mathcal S}}\O(f;Q)\chi_Q(x)$$
for almost every $x\in {\mathbb R}^n$.
\end{theorem}

As in the aforementioned work \cite{CLR26}, the proof of Theorem \ref{mr} is based on a general sparse domination principle established in \cite{LLO22}.
Observe that the condition $T(1)=0$ is a standard cancellation assumption in the theory of Calder\'on--Zygmund operators. It already appears explicitly in the classical paper of David and Journ\'e \cite{DJ84}
on the $T(1)$ theorem, where the notation $T(1)$ became one of the central notions of the subject. In the usual distributional interpretation, this condition means that $T$ annihilates constants. It is satisfied by many important Calder\'on--Zygmund operators, for instance by convolution singular integrals whose kernels have the appropriate cancellation.

We apply Theorem \ref{mr} to a question addressed in a series of recent papers by
Hoang--Moen--P\'erez \cite{HMP25_1,HMP25_2,HMP25_3}. Recall the classical
Sobolev pointwise inequality, which asserts that, for every
$f\in C^1_c({\mathbb R}^n)$, $n\ge 2$, and every $x\in{\mathbb R}^n$,
\begin{equation}\label{sob}
|f(x)|\le c_n I_1(|\nabla f|)(x),
\end{equation}
where $I_1$ is the Riesz potential operator of order $1$.

It was shown in \cite{HMP25_1,HMP25_2,HMP25_3} that $|f|$ on the
left-hand side of \eqref{sob} can be replaced by several different operators,
in particular by the Hardy--Littlewood maximal operator $M$, and by the rough
homogeneous singular integral operator $T_\Omega$ when
$\Omega\in L^{n,\infty}({\mathbb S}^{n-1})$. In \cite{HMP25_2}, the authors
asked whether such a result also holds for Calder\'on--Zygmund operators of
non-convolution type. As an almost immediate consequence of Theorem \ref{mr},
we obtain a characterization of the Calder\'on--Zygmund operators for which
this is true.

\begin{theorem}\label{sobt}
Let $T$ be a Calder\'on--Zygmund operator. Then the Sobolev-type inequality
$$
|Tf(x)|\lesssim I_1(|\nabla f|)(x)
$$
holds for every $f\in C^1_c({\mathbb R}^n)$, $n\ge 2$, and for almost every
$x\in{\mathbb R}^n$, if and only if $T(1)\in L^\infty$.
\end{theorem}

The condition $T(1)\in L^\infty$ is less standard than the classical
assumptions $T(1)=0$ and $T(1)\in BMO$, and may be viewed as intermediate
between them. In fact, it means precisely, see Proposition \ref{t1b} below, that
$T$ can be written in the form
$$
T=\widetilde T+bI,
$$
where $\widetilde T$ is a Calder\'on--Zygmund operator satisfying
$\widetilde T(1)=0$, $b\in L^\infty$, and $I$ is the identity operator.
Thus the proof of Theorem \ref{sobt} reduces to two parts: the case of
operators $\widetilde T$ satisfying $\widetilde T(1)=0$, which follows from
Theorem \ref{mr}, and the case of multiplication operators $f\mapsto bf$,
which follows immediately from \eqref{sob}.

The paper is organized as follows. Section 2 contains some preliminary, mostly known, facts. In Section 3 we prove Theorems \ref{mr} and \ref{sobt}.

\section{Preliminaries}
\subsection{Calder\'on--Zygmund operators}
We say that $T$ is a Calder\'on-Zygmund operator if $T$ is a linear operator of weak type $(1,1)$ represented by
$$Tf(x)=\int_{{\mathbb R}^n}K(x,y)f(y)dy\quad\text{for all}\,\,x\not\in \text{supp}\,f$$
with kernel $K$ satisfying the smoothness condition
\begin{equation}\label{smooth}
|K(x,y)-K(x',y)|\le \o\left(\frac{|x-x'|}{|x-y|}\right)\frac{1}{|x-y|^n}
\end{equation}
for $|x-x'|\le |x-y|/2$, where $\o$ is a modulus of continuity such that $\int_0^1\o(t)\frac{dt}{t}<\infty$.

Let us recall how the distribution $T(1)$ is defined. Denote by
${\mathscr D}_0({\mathbb R}^n)$ the space of all smooth compactly supported
functions on ${\mathbb R}^n$ with integral zero. Let
$\varphi\in {\mathscr D}_0({\mathbb R}^n)$. Choose
$\eta\in C_c^\infty({\mathbb R}^n)$ such that $\eta=1$ on a neighborhood of
$\operatorname{supp}\varphi$. Then $T(1)$ is defined by
$$
\langle T(1),\varphi\rangle
:=
\langle T(\eta),\varphi\rangle
+
\int_{{\mathbb R}^n}
\left(\int_{{\mathbb R}^n}K(x,y)\varphi(x)\,dx\right)
(1-\eta(y))\,dy .
$$
This definition is independent of the choice of $\eta$; see
\cite[Def. 4.1.16]{G14} for details.

The distribution $T(1)$ can be also defined in the following equivalent way.

\begin{prop}\label{eqdef} Let $\theta\in C^{\infty}$, $\theta(x)=1$ for $|x|<1$ and $\theta(x)=0$ for $|x|>2$. Then for $\theta_R(x):=\theta(x/R)$ we have
$$\langle T(1),\varphi\rangle=\lim_{R\to \infty}\langle T(\theta_R),\f\rangle$$
for every $\f\in {\mathscr D}_0({\mathbb R}^n)$.
\end{prop}

The proof of this proposition follows easily from the definitions, and it can be found in \cite[p. 240]{G14}.

In what follows, the condition $T(1)\in L^\infty$ means that $T(1)$ is represented by an essentially
bounded function. That is, there exists $b\in L^\infty({\mathbb R}^n)$ such that
$$
\langle T(1),\varphi\rangle
=
\int_{{\mathbb R}^n} b(x)\varphi(x)\,dx
\qquad
\text{for all } \varphi\in{\mathscr D}_0({\mathbb R}^n).
$$
Such a representative $b$ is unique only up to an additive constant.

\begin{prop}\label{t1b}
Let $T$ be a Calder\'on--Zygmund operator. Then $T(1)\in L^\infty$
if and only if
$$
T=\widetilde T+bI,
$$
where $\widetilde T$ is a Calder\'on--Zygmund operator satisfying
$\widetilde T(1)=0$, $b\in L^\infty$, and $I$ is the identity operator.
\end{prop}

\begin{proof}
Suppose that $T(1)=b$, where $b\in L^\infty$. Then define
$$
\widetilde T f:=Tf-bf.
$$
Since the multiplication operator $f\mapsto bf$ is a trivial
Calder\'on--Zygmund operator with zero off-diagonal kernel, $\widetilde T$
is again a Calder\'on--Zygmund operator. Moreover,
$$
\widetilde T(1)=T(1)-b=0
$$
in ${\mathscr D}_0'(\mathbb R^n)$.

Conversely, if $T=\widetilde T+bI$, where $\widetilde T(1)=0$ and
$b\in L^\infty$, then
$$
T(1)=\widetilde T(1)+b=b\in L^\infty.
$$
This proves the proposition.
\end{proof}

The following statement will be crucial for our purposes. Given a cube $Q$, denote $Q^*:=5{\sqrt n}Q$.

\begin{prop}\label{cr} Let $T$ be a Calder\'on--Zygmund operator satisfying $T(1)=0$. Then, for every cube $Q\subset {\mathbb R}^n$,
$$\esssup_{x',x''\in Q}|(T\chi_{Q^*})(x')-T(\chi_{Q^*})(x'')|\lesssim 1.$$
\end{prop}

\begin{proof}
For $x\in Q$, set
$$F_Q(x):=\int_{{\mathbb R}^n\setminus Q^*}\big(K(x_Q,y)-K(x,y)\big)dy,$$
where $x_Q$ is the center of $Q$. By the smoothness assumption (\ref{smooth}), the standard argument yields
\begin{equation}\label{F}
|F_Q(x)|\lesssim \sum_{k=1}^{\infty}\int_{2^kQ^*\setminus 2^{k-1}Q^*}\o\Big(\frac{\text{diam}\,Q}{|x-y|}\Big)\frac{1}{|x-y|^n}dy\lesssim \sum_{k=1}^{\infty}\o(2^{-k})\lesssim 1,
\end{equation}
uniformly in $x\in Q$.

Let us show that for every $\f\in {\mathscr D}_0$ supported in $Q$,
\begin{equation}\label{distr}
\langle T(\chi_{Q^*}),\f\rangle=\langle F_Q,\f\rangle.
\end{equation}
Once this is established, the proof follows immediately. Indeed,
(\ref{distr}) implies (see \cite[Ex. 3.2.2]{G14}) that there exists
a constant $c_Q$ such that
$$
T(\chi_{Q^*})(x)=F_Q(x)+c_Q
$$
almost everywhere on $Q$. The proof is then completed by (\ref{F}).

Turn to the proof of (\ref{distr}). Define $\theta_R$ as in Proposition \ref{eqdef}.
There exists $R_0>0$ such that for all $R\ge R_0$,
$$\theta_R=\chi_{Q^*}+\theta_R\chi_{{\mathbb R}^n\setminus Q^*},$$
and hence
$$\langle T(\chi_{Q^*}),\f\rangle=\langle T(\theta_R),\f\rangle-\langle T(\theta_R\chi_{{\mathbb R}^n\setminus Q^*}), \varphi\rangle.$$
From this, since $T(1)=0$, by Proposition \ref{eqdef},
$$\langle T(\chi_{Q^*}),\f\rangle=-\lim_{R\to\infty}\langle T(\theta_R\chi_{{\mathbb R}^n\setminus Q^*}), \varphi\rangle.$$

Next, since $\int_Q\f=0$, we have
\begin{eqnarray*}
&&\langle T(\theta_R\chi_{{\mathbb R}^n\setminus Q^*}), \varphi\rangle=\int_{{\mathbb R}^n\setminus Q^*}\theta_R(y)dy\int_{Q}K(x,y)\varphi(x)dx\\
&&=\int_{{\mathbb R}^n\setminus Q^*}\theta_R(y)dy\int_{Q}\big(K(x,y)-K(x_Q,y)\big)\varphi(x)dx.
\end{eqnarray*}
By (\ref{smooth}), the last integral is absolutely convergent uniformly in $R$. Therefore, by dominated convergence,
\begin{eqnarray*}
\langle T(\chi_{Q^*}),\f\rangle&=&-\int_{{\mathbb R}^n\setminus Q^*}dy\int_{Q}\big(K(x,y)-K(x_Q,y)\big)\varphi(x)dx\\
&=&\int_QF_Q(x)\varphi(x)dx,
\end{eqnarray*}
which proves (\ref{distr}) and completes the proof.
\end{proof}

\subsection{Rearrangements and median values}
Given a cube $Q$ and a measurable function $f$, the non-increasing rearrangement of $f$ on $Q$ is defined for $\la\in (0,1]$ by
$$(f\chi_Q)^*(\la|Q|):=\inf\big\{\a>0:|\{x\in Q:|f(x)|>\a|\le \la|Q|\big\}.$$

Define the median value of $f$ over $Q$ as a possibly non-unique number $m_f(Q)$ satisfying
$$|\{x\in Q:f(x)<m_f(Q)\}|\le \frac{|Q|}{2}$$
and
$$|\{x\in Q:f(x)>m_f(Q)\}|\le \frac{|Q|}{2}.$$

It follows from this definition and from Chebyshev's inequality that
\begin{equation}\label{medpr}
|m_f(Q)|\le \frac{2}{|Q|}\int_Q|f|.
\end{equation}

The next statement is well known. We include its proof for the sake of completeness.

\begin{prop}\label{medosc} For every $c\in {\mathbb R}$,
$$\int_Q|f-m_f(Q)|\le \int_Q|f-c|.$$
\end{prop}

\begin{proof} There exist subsets
$$A\subset \{x\in Q:f(x)\ge m_f(Q)\}\quad\text{and}\quad B\subset \{x\in Q:f(x)\le m_f(Q)\}$$
such that $|A|=|B|=\frac{1}{2}|Q|$ and $A\cap B=\emptyset$. Hence, for every $c\in {\mathbb R}$,
\begin{eqnarray*}
\int_Q|f-m_f(Q)|&=&\int_A(f-m_f(Q))+\int_B(m_f(Q)-f)=\int_Af-\int_Bf\\
&=&\int_A(f-c)+\int_B(c-f)\le \int_Q|f-c|,
\end{eqnarray*}
and the proof is complete.
\end{proof}

\subsection{Operator-free sparse domination principle} Here we describe briefly the main result of \cite{LLO22}.
For every cube $Q\subset {\mathbb R}^n$, let $f_Q$ be a measurable function. For $R\in {\mathcal D}(Q)$ set $f_{R,Q}:=f_Q-f_R$ (${\mathcal D}(Q)$ denotes a collection of all cubes dyadic with respect to $Q$).
Define the sharp maximal function
$$m_Q^{\#}f(x):=\sup_{R\in {\mathcal D}(Q):x\in R}\esssup_{x',x''\in R}|f_{R,Q}(x')-f_{R,Q}(x'')|.$$

\begin{theorem}[{\cite[Theorem 3.2]{LLO22}}]\label{opfr} Let $\la_n=2^{-n-3}$. For every $Q_0\subset {\mathbb R}^n$, there exists a $\frac{1}{2}$-sparse family ${\mathcal S}\subset {\mathcal D}(Q_0)$ such that for a.e. $x\in Q_0$,
$$|f_{Q_0}(x)|\lesssim \sum_{P\in {\mathcal S}}\Big((f_P\chi_P)^*(\la_n|P|)+(m_P^{\#}f)^*(\la_n|P|)\Big)\chi_P(x).$$
\end{theorem}

\section{Proof of Theorems \ref{mr} and \ref{sobt}}
In the proof below we use the same notation as in Proposition \ref{cr}, namely, $Q^*:=5{\sqrt n}Q$.

\begin{proof}[Proof of Theorem \ref{mr}] We will apply Theorem~\ref{opfr} with
$$f_Q(x):=T\big((f-m_f(Q^*))\chi_{Q^*}\big)(x).$$

Let $Q_0$ be an arbitrary cube. By Theorem \ref{opfr}, there exists a $\frac{1}{2}$-sparse family ${\mathcal S}\subset {\mathcal D}(Q_0)$ such that for a.e. $x\in Q_0$,
\begin{equation}\label{apopfr}
\begin{aligned}
& |T\big((f-m_f(Q_0^*))\chi_{Q_0^*}\big)(x)|\lesssim \\
& \sum_{P\in {\mathcal S}}\Big((f_P\chi_P)^*(\la_n|P|)+(m_P^{\#}f)^*(\la_n|P|)\Big)\chi_P(x).
\end{aligned}
\end{equation}

By the weak type $(1,1)$ of $T$ along with Proposition \ref{medosc},
\begin{equation}\label{firstp}
(f_P\chi_P)^*(\la_n|P|)\lesssim \frac{1}{|P|}\int_{P^*}|f-m_f(P^*)|\lesssim \O(f;P^*).
\end{equation}

We now turn to the estimate of $(m_P^{\#}f)^*(\la_n|P|)$. Let $x\in R\in {\mathcal D}(P)$. For $x',x''\in R$ consider
$$|f_{R,P}(x')-f_{R,P}(x'')|.$$
Observing that
\begin{eqnarray*}
f_{R,P}&=&T\big((f-m_f(P^*))\chi_{P^*}\big)-T\big((f-m_f(R^*))\chi_{R^*}\big)\\
&=&T\big((f-m_f(P^*))\chi_{P^*\setminus R^*}\big)+(m_f(R^*)-m_f(P^*))T(\chi_{R^*}),
\end{eqnarray*}
we have
\begin{eqnarray*}
&&|f_{R,P}(x')-f_{R,P}(x'')|\\
&&\le |T\big((f-m_f(P^*))\chi_{P^*\setminus R^*}\big)(x')-T\big((f-m_f(P^*))\chi_{P^*\setminus R^*}\big)(x'')|\\
&&+|m_f(R^*)-m_f(P^*)||T(\chi_{R^*})(x')-T(\chi_{R^*})(x'')|.
\end{eqnarray*}

By the smoothness condition (\ref{smooth}),
\begin{eqnarray}
&&|T\big((f-m_f(P^*))\chi_{P^*\setminus R^*}\big)(x')-T\big((f-m_f(P^*))\chi_{P^*\setminus R^*}\big)(x'')|\label{estpart}\\
&&\lesssim \int_{{\mathbb R}^n\setminus R^*}|(f-m_f(P^*))\chi_{P^*}(y)|\o\Big(\frac{\text{diam}\,R}{|x'-y|}\Big)\frac{1}{|x'-y|^n}dy\nonumber\\
&&\lesssim\sum_{k=1}^{\infty}\o(2^{-k})\frac{1}{|2^kR^*|}\int_{2^kR^*\setminus 2^{k-1}R^*}|(f-m_f(P^*))\chi_{P^*}(y)|dy\nonumber\\
&&\lesssim \Big(\int_0^1\o(t)\frac{dt}{t}\Big)M\big((f-m_f(P^*))\chi_{P^*}\big)(x).\nonumber
\end{eqnarray}

Next, by (\ref{medpr}),
\begin{eqnarray*}
|m_f(R^*)-m_f(P^*)|=|m_{f-m_f(P^*)}(R^*)|&\le& \frac{2}{|R^*|}\int_{R^*}|f-m_f(P^*)|\\
&\le& 2M\big((f-m_f(P^*))\chi_{P^*}\big)(x).
\end{eqnarray*}
Further, by Proposition \ref{cr},
$$
\esssup_{x',x''\in R}|T(\chi_{R^*})(x')-T(\chi_{R^*})(x'')|\lesssim 1,
$$
which, along with the previous estimate and (\ref{estpart}), implies
$$m_P^{\#}f(x)\lesssim M\big((f-m_f(P^*))\chi_{P^*}\big)(x).$$

From this, by the weak type $(1,1)$ of $M$ along with Proposition \ref{medosc},
$$(m_P^{\#}f)^*(\la_n|P|)\lesssim \frac{1}{|P|}\int_{P^*}|f-m_f(P^*)|\lesssim \O(f;P^*).$$
Combining this with (\ref{apopfr}) and (\ref{firstp}), we obtain that there exists a $\frac{1}{2}$-sparse family ${\mathcal S}\subset {\mathcal D}(Q_0)$ such that for a.e. $x\in Q_0$,
\begin{equation}\label{local}
|T\big((f-m_f(Q_0^*))\chi_{Q_0^*}\big)(x)|\lesssim \sum_{P\in {\mathcal S}}\O(f;P^*)\chi_{P}(x).
\end{equation}

Now, the rest of the proof follows the same lines as that of \cite[Lemma 2.1]{LO20}. Take a partition of ${\mathbb R}^n$ by cubes $Q_j$ such that $\text{supp}\,f\subset Q_j^*$, and moreover,
$|\text{supp}\,f|<\frac{1}{2}|Q_j^*|$ for every $j$.
For example, take a cube $S$ such that
$\text{supp}\,(f)\subset S$ and cover $3S\setminus S$ by $3^n-1$ congruent cubes $Q_j$. Each of them satisfies $S\subset 3Q_j\subset Q_j^*$.
Next, in the same way cover $9S\setminus 3S$, and so on. The union of resulting cubes, including $S$, will satisfy the desired properties.

Since $|\text{supp}\,f|<\frac{1}{2}|Q_j^*|$, we have $m_f(Q_j^*)=0$ for every $j$. Therefore, applying (\ref{local}) to every $Q_j$ instead of $Q_0$, we obtain a $\frac{1}{2}$-sparse family ${\mathcal F}_j\subset {\mathcal D}(Q_j)$ such that for almost every $x\in Q_j$,
$$|T(f)(x)|\lesssim \sum_{P\in {\mathcal F}_j}\O(f;P^*)\chi_{P}(x).$$
From this, setting ${\mathcal F}:=\cup_j{\mathcal F}_j$, we obtain that ${\mathcal F}$ is $\frac{1}{2}$-sparse and
$$|T(f)(x)|\lesssim \sum_{P\in {\mathcal F}}\O(f;P^*)\chi_{P}(x)$$
for almost every $x\in {\mathbb R}^n$. This finishes the proof with a $\frac{1}{2\cdot (5\sqrt n)^n}$-sparse family ${\mathcal S}:=\{P^*:P\in {\mathcal F}\}$.
\end{proof}

\begin{proof}[Proof of Theorem \ref{sobt}]
Suppose first that $T(1)\in L^{\infty}$. Then, by Proposition \ref{t1b} and by the Sobolev inequality (\ref{sob}), the proof reduces to showing that
\begin{equation}\label{suf}
|Tf(x)|\lesssim I_1(|\nabla f|)(x)
\end{equation}
for $T$ satisfying $T(1)=0$.

Assume that $T(1)=0$. Combining Theorem \ref{mr} with the $(1,1)$ Poincar\'e inequality
$$\O(f;Q)\lesssim \frac{1}{|Q|^{1-1/n}}\int_Q|\nabla f|,$$
we obtain
$$
|Tf(x)|\lesssim \sum_{Q\in {\mathcal S}}\Big(\frac{1}{|Q|^{1-1/n}}\int_Q|\nabla f|\Big)\chi_Q(x).
$$
Observe that the cubes in ${\mathcal S}$ here are not necessarily dyadic. However, by the three lattice theorem (see, e.g., \cite[Th. 3.1]{LN19}), the above estimate implies that there are $3^n$ dyadic lattices ${\mathscr D}_j$ and $\eta_n$-sparse families ${\mathcal S}_j\subset {\mathscr D}_j$ such that
\begin{equation}\label{dv}
|Tf(x)|\lesssim \sum_{j=1}^{3^n}\sum_{Q\in {\mathcal S}_j}\Big(\frac{1}{|Q|^{1-1/n}}\int_Q|\nabla f|\Big)\chi_Q(x).
\end{equation}

On the other hand, it is well-known \cite[Prop. 2.2]{CUM13} that for every dyadic lattice ${\mathscr D}$ and any non-negative locally integrable $g$,
$$\sum_{Q\in {\mathscr D}}\Big(\frac{1}{|Q|^{1-1/n}}\int_Qg\Big)\chi_Q(x)\lesssim I_1g(x).$$
This, along with (\ref{dv}), proves (\ref{suf}).

Suppose now that $T$ is a Calder\'on--Zygmund operator such that (\ref{suf}) holds for every $f\in C^1_c({\mathbb R}^n)$, $n\ge 2$, and for almost every
$x\in{\mathbb R}^n$. Let us show that $T(1)\in~L^{\infty}$.

Let $\theta_R$ be defined as in Proposition \ref{eqdef}. Observe that
$$
I_1(|\nabla\theta_R|)(x)=\int_{{\mathbb R}^n}\frac{|\nabla \theta(y)|}{|x/R-y|^{n-1}}dy.
$$
Since $\nabla\theta$ is bounded and with compact support, we have that this integral is uniformly bounded in $x$ and $R$.
Therefore, by (\ref{suf}),
$$
\|T(\theta_R)\|_{L^\infty}\lesssim 1
$$
uniformly in $R$. Hence, by Proposition \ref{eqdef}, for every
$\varphi\in{\mathscr D}_0({\mathbb R}^n)$,
$$
|\langle T(1),\varphi\rangle|
=
\lim_{R\to\infty}|\langle T(\theta_R),\varphi\rangle|
\le C\|\varphi\|_{L^1}.
$$
Thus $T(1)$ is bounded on ${\mathscr D}_0({\mathbb R}^n)$ with respect to the
$L^1$-norm. By the Hahn--Banach theorem and the duality
$(L^1)^*=L^\infty$, there exists $b\in L^\infty({\mathbb R}^n)$ such that
$$
\langle T(1),\varphi\rangle
=
\int_{{\mathbb R}^n} b(x)\varphi(x)\,dx
\qquad
\text{for all }\varphi\in{\mathscr D}_0({\mathbb R}^n).
$$
Therefore $T(1)\in L^\infty$, and the proof is complete.
\end{proof}

\end{document}